\documentclass[12pt,reqno,a4paper]{amsart}
\usepackage{amsfonts}
\usepackage{amsmath}

\newtheorem{theorem}{Theorem}[section]
\newtheorem{corollary}[theorem]{Corollary}
\newtheorem{lemma}[theorem]{Lemma}
\theoremstyle{definition}

\theoremstyle{remark}

\input tcilatex

\begin{document}
\title[On certain subclasses of analytic functions associated with ...]{On
certain subclasses of analytic functions ASSOCIATED WITH POISSON
DISTRIBUTION SERIES}
\author{B.A. Frasin}
\address{Department of Mathematics, Al al-Bayt University, Mafraq, Jordan}
\email{bafrasin@yahoo.com}

\begin{abstract}
In this paper, we find the necessary and sufficient conditions, inclusion
relations for Poisson distribution series $\mathcal{K}(m,z)=z+\sum%
\limits_{n=2}^{\infty }\frac{m^{n-1}}{(n-1)!}e^{-m}z^{n}$ belonging to the
subclasses $\mathcal{S}(k,\lambda )$ and $\mathcal{C}(k,\lambda )$ of
analytic functions with negative coefficients. Further, we consider the
integral operator $\mathcal{G(}m,z\mathcal{)=}\dint\limits_{0}^{z}\frac{%
\mathcal{F}(m,t)}{t}dt$ belonging to the above classes.


\end{abstract}

\maketitle

\section{\protect\bigskip Introduction and definitions}

Let $\mathcal{A}$ denote the class of the normalized functions of the form%
\begin{equation}
f(z)=z+\sum_{n=2}^{\infty }a_{n}z^{n},
\end{equation}%
which are analytic in the open unit disk $\mathcal{U}=\{z\in \in \mathbb{C}%
:\left\vert z\right\vert <1\}.$ Further, let $\mathcal{T}$ \ be a subclass
of $\mathcal{A}$ consisting of functions of the form, 
\begin{equation}
f(z)=z-\sum\limits_{n=2}^{\infty }\left\vert a_{n}\right\vert z^{n},\qquad
z\in \mathcal{U}\text{.}  \label{m1}
\end{equation}

\noindent A function\ $f$\ of the form (\ref{m1}) is in $\mathcal{S}%
(k,\lambda )$ if it satisfies the condition%
\begin{equation*}
\left\vert \frac{\frac{zf^{\prime }(z)}{(1-\lambda )f(z)+\lambda zf^{\prime
}(z)}-1}{\frac{zf^{\prime }(z)}{(1-\lambda )f(z)+\lambda zf^{\prime }(z)}+1}%
\right\vert <k,\ \ (0<k\leq 1,\text{ }0\leq \lambda <1,z\in \mathcal{U})
\end{equation*}%
and\textit{\ }$f\in $\textit{\ }$\mathcal{C}(k,\lambda )$\textit{\ }if and
only if $zf^{\prime }\in \mathcal{S}(k,\lambda ).$The class $\mathcal{S}%
(k,\lambda )$ was introduced by Frasin et al. \cite{fra}.

We note that $\mathcal{S}(k,0)=\mathcal{S}(k)$ and $\mathcal{C}(k,0)=%
\mathcal{C}(k)$, where the classes $\mathcal{S}(k)$ and $\mathcal{C}(k)$
were introduced and studied by Padmanabhan \cite{pad} (see also, \cite{mog}, 
\cite{owa}) .

A function $f\in \mathcal{A}$ is said to be in the class $\mathcal{R}^{\tau
}(A,B)$,$\tau \in \mathbb{C}\backslash \{0\}$, $-1\leq B<A\leq 1,$ if it
satisfies the inequality%
\begin{equation*}
\left\vert \frac{f^{\prime }(z)-1}{(A-B)\tau -B[f^{\prime }(z)-1]}%
\right\vert <1,\ \ \ \;z\in \mathcal{U}.
\end{equation*}

\bigskip This class was introduced by Dixit and Pal \cite{dix}.

\bigskip A variable $X$ is said to be Poisson distributed if it takes the
values $0,1,2,3,\cdots $ with probabilities $e^{-m}$, $m\frac{e^{-m}}{1!}$, $%
m^{2}\frac{e^{-m}}{2!}$, $m^{3}\frac{e^{-m}}{3!},\cdots $ respectively,
where $m$ is called the parameter. Thus

\begin{equation*}
P(X=r)=\frac{m^{r}e^{-m}}{r!},\text{ }r=0,1,2,3,\cdots .
\end{equation*}

Very recently, Porwal \cite{por1} (see also, \cite{mur1,mur2}) introduce a
power series whose coefficients are probabilities of Poisson distribution%
\begin{equation*}
\mathcal{K}(m,z)=z+\sum\limits_{n=2}^{\infty }\frac{m^{n-1}}{(n-1)!}%
e^{-m}z^{n},\text{\ \ \ \ \ \ \ \ }z\in \mathcal{U}\text{,}
\end{equation*}

where $m>0.$ By ratio test the radius of convergence of above series is
infinity. In \cite{por1}, Porwal also defined the series

\begin{equation*}
\mathcal{F}(m,z)=2z-\mathcal{K}(m,z)=z-\sum\limits_{n=2}^{\infty }\frac{%
m^{n-1}}{(n-1)!}e^{-m}z^{n},\text{\ \ \ \ \ \ \ \ }z\in \mathcal{U}\text{.}
\end{equation*}

Using the Hadamard product, Porwal and Kumar \cite{por2} introduced a new
linear operator $\mathcal{I}(m,z):\mathcal{A\rightarrow A}$ defined by%
\begin{equation*}
\mathcal{I}(m,z)f=\mathcal{K}(m,z)\ast f(z)=z+\sum\limits_{n=2}^{\infty }%
\frac{m^{n-1}}{(n-1)!}e^{-m}a_{n}z^{n},\text{\ \ \ \ \ \ \ \ }z\in \mathcal{U%
}\text{,}
\end{equation*}%
where $\ast $ denote the convolution or Hadamard product of two series.

Motivated by several earlier results on connections between various
subclasses of analytic and univalent functions by using hypergeometric
functions (see \cite{1,2,8,9}) and by the recent investigations of Porwal (%
\cite{por1, por2, por3}), in the present paper we determine the necessary
and sufficient conditions for $\mathcal{F}(m,z)$ to be in our new classes $%
\mathcal{S}(k,\lambda )$ and $\mathcal{C}(k,\lambda )$ and connections of
these subclasses with $\mathcal{R}^{\tau }(A,B)$. Finally, we give
conditions for the integral operator $\mathcal{G(}m,z\mathcal{)=}%
\dint\limits_{0}^{z}\frac{\mathcal{F}(m,t)}{t}dt$ belonging to the classes $%
\mathcal{S}(k,\lambda )$ and $\mathcal{C}(k,\lambda ).$

To establish our main results, we will require the following Lemmas.

\begin{lemma}
\label{lem1}\cite{fra} A function\ $f$\ of the form (\ref{m1}) is in $%
\mathcal{S}(k,\lambda )$ if and only if it satisfies%
\begin{equation}
\sum\limits_{n=2}^{\infty }[n((1-\lambda )+k(1+\lambda ))-(1-\lambda
)(1-k)]\left\vert a_{n}\right\vert \leq 2k  \label{b1}
\end{equation}
\end{lemma}

\textit{where }$0<k\leq 1$ \textit{\ and} $0\leq \lambda <1.$ \textit{The
result is sharp.}

\begin{lemma}
\label{lem2}\cite{fra}\textit{\ A function\ }$f$\textit{\ of the form (\ref%
{m1}) is in }$\mathcal{C}(k,\lambda )$\textit{\ if and only if it satisfies}
\end{lemma}

\begin{equation}
\sum\limits_{n=2}^{\infty }n[n((1-\lambda )+k(1+\lambda ))-(1-\lambda
)(1-k)]\left\vert a_{n}\right\vert \leq 2k  \label{h}
\end{equation}

\textit{where }$0<k\leq 1$ \textit{\ and} $0\leq \lambda <1.$ \textit{The
result is sharp.}\ 

\begin{lemma}
\label{lem3}\cite{dix}If $f$\textit{\ }$\in $\textit{\ }$\mathcal{R}^{\tau
}(A,B)$ is of the form , then%
\begin{equation*}
\left\vert a_{n}\right\vert \leq (A-B)\frac{\left\vert \tau \right\vert }{n}%
,\ \ \ \ \ \ n\in \mathbb{N}-\{1\}\text{.}
\end{equation*}%
\textit{The result is sharp.}
\end{lemma}

\section{The necessary and sufficient conditions}

\begin{theorem}
\label{th1}If $m>0$, $0<k\leq 1$ \textit{\ and} $0\leq \lambda <1,$ then $%
\mathcal{F}(m,z)$ is in $\mathcal{S}(k,\lambda )$ if and only if 
\begin{equation}
((1-\lambda )+k(1+\lambda ))me^{m}\leq 2k.  \label{d1}
\end{equation}
\end{theorem}

\begin{proof}
Since%
\begin{equation}
\mathcal{F}(m,z)=z-\sum\limits_{n=2}^{\infty }\frac{m^{n-1}}{(n-1)!}%
e^{-m}z^{n}  \label{n5}
\end{equation}%
according to (\ref{b1}) of Lemma \ref{lem1}, we must show that%
\begin{equation}
\sum\limits_{n=2}^{\infty }[n((1-\lambda )+k(1+\lambda ))+(1-\lambda )(k-1)]%
\frac{m^{n-1}}{(n-1)!}\leq 2ke^{m}.  \label{n1}
\end{equation}%
Writing $n=(n-1)+1$, we have 
\begin{eqnarray}
&&\sum\limits_{n=2}^{\infty }[n((1-\lambda )+k(1+\lambda ))+(1-\lambda
)(k-1)]\frac{m^{n-1}}{(n-1)!}  \notag \\
&=&\ \sum\limits_{n=2}^{\infty }[(n-1)((1-\lambda )+k(1+\lambda ))+2k]\frac{%
m^{n-1}}{(n-1)!}\   \notag \\
&=&[(1-\lambda )+k(1+\lambda )]\ \sum\limits_{n=2}^{\infty }\frac{m^{n-1}}{%
(n-2)!}+2k\sum\limits_{n=2}^{\infty }\frac{m^{n-1}}{(n-1)!}  \notag \\
&=&\ \left( (1-\lambda )+k(1+\lambda ))me^{m}+2k\allowbreak (e^{m}-1\right) .
\label{oo}
\end{eqnarray}%
But this last expression is bounded above by $2ke^{m}~$if and only if~(\ref%
{d1}) holds.
\end{proof}

\begin{theorem}
\label{th2}If $m>0$, $0<k\leq 1$ \textit{\ and} $0\leq \lambda <1,$ then $%
\mathcal{F}(m,z)$ is in $\mathcal{C}(k,\lambda )$ if and only if 
\begin{equation}
((1-\lambda )+k(1+\lambda ))m^{2}e^{m}+2(1+2k+k\lambda -\lambda )me^{m}\leq
2k  \label{d2}
\end{equation}
\end{theorem}

\begin{proof}
In view of Lemma \ref{lem2}, it suffices to show that%
\begin{equation*}
\sum\limits_{n=2}^{\infty }n[n((1-\lambda )+k(1+\lambda ))+(1-\lambda )(k-1)]%
\frac{m^{n-1}}{(n-1)!}\leq 2ke^{m}.
\end{equation*}%
Now%
\begin{eqnarray}
&&\sum\limits_{n=2}^{\infty }n[n((1-\lambda )+k(1+\lambda ))+(1-\lambda
)(k-1)]\frac{m^{n-1}}{(n-1)!}  \notag \\
&=&\sum\limits_{n=2}^{\infty }n^{2}((1-\lambda )+k(1+\lambda ))+n(1-\lambda
)(k-1)]\frac{m^{n-1}}{(n-1)!}.  \label{v}
\end{eqnarray}%
Writing $n=(n-1)+1$ and $n^{2}=(n-1)(n-2)+3(n-1)+1$, in (\ref{v}) we see that%
\begin{eqnarray*}
&&\sum\limits_{n=2}^{\infty }n^{2}((1-\lambda )+k(1+\lambda ))+n(1-\lambda
)(k-1)]\frac{m^{n-1}}{(n-1)!} \\
&=&\sum\limits_{n=2}^{\infty }(n-1)(n-2)((1-\lambda )+k(1+\lambda ))\frac{%
m^{n-1}}{(n-1)!} \\
&&+\sum\limits_{n=2}^{\infty }\ (n-1)[3((1-\lambda )+k(1+\lambda
)+(1-\lambda )(k-1)]\frac{m^{n-1}}{(n-1)!}+\sum\limits_{n=2}^{\infty }2k%
\frac{m^{n-1}}{(n-1)!}\ \  \\
&=&((1-\lambda )+k(1+\lambda ))\sum\limits_{n=2}^{\infty }\frac{m^{n-1}}{%
(n-3)!}\ +2(1+2k+k\lambda -\lambda )\sum\limits_{n=2}^{\infty }\frac{m^{n-1}%
}{(n-2)!} \\
&&+2k\sum\limits_{n=2}^{\infty }\frac{m^{n-1}}{(n-1)!} \\
&=&((1-\lambda )+k(1+\lambda ))m^{2}e^{m}+2(1+2k+k\lambda -\lambda
)me^{m}+2k\ (e^{m}-1).\ \ \ \ \ \ \ \ \ \ \ \ \ \ \ \ \ \ \ \ \ \ \ \ \ \ \
\ \ \ \ \ \ \ \ \ \ \ \ \ 
\end{eqnarray*}

But this last expression is bounded above by $2ke^{m}~$if and only if~(\ref%
{d2}) holds.
\end{proof}

By specializing the parameter $\lambda =0$ in Theorems \ref{th1} and \ref%
{th2} , we have the following corollaries.

\begin{corollary}
If $m>0$\textit{\ and} $0<k\leq 1,$ then $\mathcal{F}(m,z)$ is in $\mathcal{S%
}(k)$ if and only if 
\begin{equation}
(1+k)me^{m}\leq 2k.  \label{tt}
\end{equation}
\end{corollary}

\begin{corollary}
If $m>0$\textit{\ and} $0<k\leq 1,$ then $\mathcal{F}(m,z)$ is in $\mathcal{C%
}(k)$ if and only if 
\begin{equation}
(1+k)m^{2}e^{m}+2(1+2k)me^{m}\leq 2k.
\end{equation}
\end{corollary}

\section{Inclusion Properties}

\begin{theorem}
\label{th5}Let $m>0$, $0<k\leq 1$ \textit{\ and} $0\leq \lambda <1.$ If $\
f\in \mathcal{R}^{\tau }(A,B),\ $then $\mathcal{I}(m,z)f$ is in $\mathcal{S}%
(k,\lambda )$ if and only if 
\begin{equation}
(A-B)\left\vert \tau \right\vert \left[ ((1-\lambda )+k(1+\lambda
))(1-e^{-m})+\frac{(1-\lambda )(k-1)}{m}(1-e^{-m}(1+m))\right] \leq 2k.
\label{d3}
\end{equation}
\end{theorem}

\begin{proof}
In view of Lemma \ref{lem1}, it suffices to show that%
\begin{equation*}
\sum\limits_{n=2}^{\infty }[n((1-\lambda )+k(1+\lambda ))+(1-\lambda )(k-1)]%
\frac{m^{n-1}}{(n-1)!}\left\vert a_{n}\right\vert \leq 2ke^{m}.
\end{equation*}

Since $f\in \mathcal{R}^{\tau }(A,B),$ then by Lemma \ref{lem3}, we get%
\begin{equation}
\left\vert a_{n}\right\vert \leq \frac{(A-B)\left\vert \tau \right\vert }{n}.
\label{vv}
\end{equation}

Thus, we have 
\begin{eqnarray*}
&&\sum\limits_{n=2}^{\infty }[n((1-\lambda )+k(1+\lambda ))+(1-\lambda
)(k-1)]\frac{m^{n-1}}{(n-1)!}\left\vert a_{n}\right\vert \\
&\leq &(A-B)\left\vert \tau \right\vert \sum\limits_{n=2}^{\infty
}[n((1-\lambda )+k(1+\lambda ))+(1-\lambda )(k-1)]\frac{m^{n-1}}{n!} \\
&=&(A-B)\left\vert \tau \right\vert \left[ ((1-\lambda )+k(1+\lambda
))\sum\limits_{n=2}^{\infty }\frac{m^{n-1}}{(n-1)!}+\frac{(1-\lambda )(k-1)}{%
m}\sum\limits_{n=2}^{\infty }\frac{m^{n}}{n!}\right] \\
&=&(A-B)\left\vert \tau \right\vert \left[ ((1-\lambda )+k(1+\lambda
))(e^{m}-1)+\frac{(1-\lambda )(k-1)}{m}(e^{m}-1-m)\right] .
\end{eqnarray*}

But this last expression is bounded above by $2ke^{m}~$if and only if~(\ref%
{d3}) holds.
\end{proof}

\begin{theorem}
\label{th6}Let $m>0$, $0<k\leq 1$ \textit{\ and} $0\leq \lambda <1.$ If $\
f\in \mathcal{R}^{\tau }(A,B),\ $then $\mathcal{F}(m,z)f$ is in $\mathcal{C}%
(k,\lambda )$ if and only if 
\begin{equation}
(A-B)\left\vert \tau \right\vert [((1-\lambda )+k(1+\lambda
))m+2k(1-e^{-m})]\leq 2k.  \label{d5}
\end{equation}
\end{theorem}

\begin{proof}
In view of Lemma \ref{lem2}, it suffices to show that%
\begin{equation*}
\sum\limits_{n=2}^{\infty }n[n((1-\lambda )+k(1+\lambda ))+(1-\lambda )(k-1)]%
\frac{m^{n-1}}{(n-1)!}\left\vert a_{n}\right\vert \leq 2ke^{m}.
\end{equation*}

Using (\ref{vv}), we have%
\begin{eqnarray*}
&&\sum\limits_{n=2}^{\infty }n[n((1-\lambda )+k(1+\lambda ))+(1-\lambda
)(k-1)]\frac{m^{n-1}}{(n-1)!}\left\vert a_{n}\right\vert \\
&\leq &\sum\limits_{n=2}^{\infty }n[n((1-\lambda )+k(1+\lambda ))+(1-\lambda
)(k-1)]\frac{m^{n-1}}{(n-1)!}\frac{(A-B)\left\vert \tau \right\vert }{n} \\
&=&(A-B)\left\vert \tau \right\vert \sum\limits_{n=2}^{\infty }[n((1-\lambda
)+k(1+\lambda ))+(1-\lambda )(k-1)]\frac{m^{n-1}}{(n-1)!} \\
&=&(A-B)\left\vert \tau \right\vert \sum\limits_{n=2}^{\infty
}[(n-1)((1-\lambda )+k(1+\lambda ))+2k]\frac{m^{n-1}}{(n-1)!} \\
&=&(A-B)\left\vert \tau \right\vert \left[ ((1-\lambda )+k(1+\lambda
))\sum\limits_{n=2}^{\infty }\frac{m^{n-1}}{(n-2)!}+2k\sum\limits_{n=2}^{%
\infty }\frac{m^{n-1}}{(n-1)!}\right] \\
&=&(A-B)\left\vert \tau \right\vert [((1-\lambda )+k(1+\lambda
))me^{m}+2k(e^{m}-1)].
\end{eqnarray*}

But this last expression is bounded above by $2ke^{m}~$if and only if~(\ref%
{d5}) holds.
\end{proof}

By taking $\lambda =0$ in Theorems \ref{th5} and \ref{th6} , we obtain the
following corollaries.

\begin{corollary}
Let $m>0$\textit{\ and} $0<k\leq 1.$ If $\ f\in \mathcal{R}^{\tau }(A,B),\ $%
then $\mathcal{I}(m,z)f$ is in $\mathcal{S}(k)$ if and only if 
\begin{equation}
(A-B)\left\vert \tau \right\vert \left[ (1+k)(1-e^{-m})+\frac{(k-1)}{m}%
(1-e^{-m}(1+m))\right] \leq 2k.
\end{equation}
\end{corollary}

\begin{corollary}
Let $m>0$\textit{\ and} $0<k\leq 1.$ If $\ f\in \mathcal{R}^{\tau }(A,B),\ $%
then $\mathcal{I}(m,z)f$ is in $\mathcal{C}(k)$ if and only if 
\begin{equation}
(A-B)\left\vert \tau \right\vert [(1+k)m+2k(1-e^{-m})]\leq 2k.
\end{equation}
\end{corollary}

\section{An integral operator}

\begin{theorem}
\label{th3}If $m>0$, $0<k\leq 1$ \textit{\ and} $0\leq \lambda <1$, then 
\begin{equation}
\mathcal{G(}m,z\mathcal{)=}\dint\limits_{0}^{z}\frac{\mathcal{F}(m,t)}{t}dt
\label{pl}
\end{equation}%
is in $\mathcal{C}(k,\lambda )$ if and only if ~(\ref{d1}) is satisfied.
\end{theorem}

\begin{proof}
Since%
\begin{equation*}
\mathcal{G(}m,z\mathcal{)}=z-\sum\limits_{n=2}^{\infty }\frac{e^{-m}m^{n-1}}{%
n!}z^{n}
\end{equation*}%
then by Lemma \ref{lem2}, we need only to show that%
\begin{equation*}
\sum\limits_{n=2}^{\infty }n[n((1-\lambda )+k(1+\lambda ))+(1-\lambda )(k-1)]%
\frac{m^{n-1}}{n!}\leq 2ke^{m}.
\end{equation*}

or, equivalently%
\begin{equation*}
\sum\limits_{n=2}^{\infty }[n((1-\lambda )+k(1+\lambda ))+(1-\lambda )(k-1)]%
\frac{m^{n-1}}{(n-1)!}\leq 2ke^{m}.
\end{equation*}

From (\ref{oo}) it follows that%
\begin{eqnarray*}
&&\sum\limits_{n=2}^{\infty }[n((1-\lambda )+k(1+\lambda ))+(1-\lambda
)(k-1)]\frac{m^{n-1}}{(n-1)!} \\
&=&\ \left( (1-\lambda )+k(1+\lambda ))me^{m}+2k\allowbreak (e^{m}-1\right)
\end{eqnarray*}

and this last expression is bounded above by $2ke^{m}~$if and only if~(\ref%
{d1}) holds.
\end{proof}

The proof of Theorem \ref{th4} below is much similar to that of Theorem \ref%
{th3} and so the details are omitted.

\begin{theorem}
\label{th4}If $m>0$, $0<k\leq 1$ \textit{\ and} $0\leq \lambda <1$, then the
integral operator defined by(\ref{pl}) is in $\mathcal{S}(k,\lambda )$ if
and only if ~%
\begin{equation*}
((1-\lambda )+k(1+\lambda ))(1-e^{-m})+\frac{(1-\lambda )(k-1)}{m}%
(1-e^{-m}-me^{-m})\leq 2k.
\end{equation*}
\end{theorem}

By taking $\lambda =0$ in Theorems \ref{th3} and \ref{th4} , we obtain the
following corollaries.

\begin{corollary}
If $m>0$\textit{\ and} $\ 0<k\leq 1$, then the integral operator defined by(%
\ref{pl}) is in $\mathcal{C}(k)$ if and only if ~(\ref{tt}) is satisfied.
\end{corollary}

\begin{corollary}
If $m>0$\textit{\ and} $\ 0<k\leq 1$,then the integral operator defined by(%
\ref{pl}) is in $\mathcal{S}(k)$ if and only if ~%
\begin{equation*}
(1+k)(1-e^{-m})+\frac{(k-1)}{m}(1-e^{-m}-me^{-m})\leq 2k.
\end{equation*}
\end{corollary}

\end{document}